\documentclass[12pt,a4paper,twoside,english]{smfart}

\usepackage{bbold}
\usepackage{lmodern}
\usepackage[utf8]{inputenc}
\usepackage[T1]{fontenc}
\usepackage{amssymb}
\usepackage{amsmath}
\usepackage{amsthm}
\usepackage{amsfonts}
\usepackage{graphicx}
\usepackage[colorlinks=true,linkcolor=black,citecolor=black,urlcolor=black]{hyperref}
\usepackage{fancyhdr}
\usepackage[all,cmtip]{xy}
\usepackage{alltt}
\usepackage{footmisc}
\usepackage{smfthm}

\usepackage{calrsfs}

\usepackage{fullpage}

\xyoption{all}

\xyoption{frame}

\newcommand{\finpreuve}{\mbox{} \hfill \mbox{$\Box$}}

\DeclareUnicodeCharacter{00A0}{~} 

\def\K{{\rm K}}
\def\L{{\rm L}}
\def\F{{\rm F}}

\def\Q{{\mathbb Q}}
\def\Z{{\mathbb Z}}
\def\fq{{\mathbb F}}

\def\Nat{{\mathbb N}}
\def\PP{{\mathbb P}}

\def\N{{\rm N}}
\def\R{{\rm R}}
\def\d{{\rm d}}

\def\C{{\mathcal C}}
\def\H{{\mathcal H}}
\def\G{{\mathcal G}}

\def\X{{\mathcal X}}

\def\O{{\mathcal O}}

\def\U{{\mathcal U}}

\def\E{{\mathcal E}}
\def\M{{\mathcal M}}
\def\Ff{{\mathcal F}}
\def\To{{\mathcal T}}

\def\ll{{\mathcal L}}

\def\p{{\mathfrak p}}
\def\P{{\mathfrak P}}

\def\q{{\mathfrak q}}
\def\mm{{\mathfrak m}}

\def\Irr{{\rm Irr}}

\def\Gl{{\rm Gl}}

\def\MM{{\rm M}}

\def\Gal{{\rm Gal}}
\def\rg{{\rm rk}}
\def\cd{{\rm cd}}
\def\ab{{\mathrm{ab}}}
\def\Aut{{\rm Aut}}
\def\cd{{\rm cd}}
\def\inert{{\rm inert}}
\def\split{{\rm split}}

\def\GST{{\G_S^T}}
\def\Frob{{\rm Frob}}
\def\Fix{{\rm Fix}}

\def\A{{\rm A}}

\def\mod{{\rm mod}}
\def\ker{{\rm ker}}
\def\Res{{\rm Res}}
\def\Ind{{\rm Ind}}

\def\1{{\bf 1}}

\def\ACT#1#2{{\displaystyle{#1^{\circlearrowleft^#2}}}}

\newcommand{\bigO}[1]{\ensuremath{\mathop{}\mathopen{}O\mathopen{}\left(#1\right)}}

\def\act#1{{\displaystyle{#1^{\circlearrowleft^\sigma}}}}

\newenvironment{Question}{\begin{enonce}{Question}}{\end{enonce}}
\newenvironment{conjecture}{\begin{enonce}{Conjecture}}{\end{enonce}}

\begin{document}

\date{\today}

\author{Farshid Hajir}
\address{Department of Mathematics \& Statistics, University of Massachusetts, Amherst MA 01003, USA.}

\author{Christian Maire}
\address{Laboratoire de Math\'ematiques, UMR 6623 Universit\'e de Franche-Comté et CNRS, 16 route de Gray, 25030 Besan\c con, France}


\thanks{\emph{Acknowledgements.} The second author would like to thank the Department of Mathematics \& Statistics at UMass Amherst for its hospitality during several visits.  He also thanks UMI CNRS  3457  at Montréal for providing a stimulating research atmosphere on a "Délégation" there.
The authors are grateful to Georges Gras for helpful correspondence regarding $p$-rational fields.
}

\title{Prime Decomposition and the Iwasawa MU-invariant}

\begin{abstract}
For $\Gamma=\Z_p$, Iwasawa was the first to construct $\Gamma$-extensions over number fields with arbitrarily large $\mu$-invariants.  
In this work, we investigate other uniform pro-$p$ groups which are realizable
 as Galois groups of towers of number fields with arbitrarily large $\mu$-invariant. For instance, we prove that this is the case if $p$ is a regular prime and $\Gamma$ is a uniform pro-$p$ group   admitting a fixed-point-free automorphism of odd order dividing $p-1$.  Both in Iwasawa's work, and in the present one, the size of the $\mu$-invariant appears to be intimately related to the existence of primes that split completely in the tower.
\end{abstract}

\maketitle

\tableofcontents

\section*{Introduction}


\

Let $p$ be a prime number.
 Let $\K$ be a number field and let $\L/\K$ be a uniform $p$-extension: $\L/\K$ is a normal extension whose Galois group $\Gamma:=\Gal(\L/\K)$  is a uniform pro-$p$ group (see section \ref{sectionuniform}). 
We suppose moreover  that \emph{the set of places of $\K$  that are ramified in $\L/\K$ is finite}. 
 
 \medskip
 
 If $\F/\K$ is a finite subextension of $\L/\K$, let us denote by $\A(\F)$ the $p$-Sylow subgroup of the class group of $\F$ and put $$\X:=\lim_{\stackrel{\leftarrow}{\F}} \A(\F),$$
where the limit is taken over all number fields $\F$ in $\L/\K$ with respect the norm map.
Then $\X$ is a $\Z_p[[\Gamma]]$-module and, thanks to a structure theorem (see section \ref{sectionarithmetic}), one attaches a $\mu$-invariant to $\X$, generalizing the well-known $\mu$-invariant introduced by Iwasawa in the classical case  $\Gamma\simeq \Z_p$.  Iwasawa showed that the size of the $\mu$-invariant is related to the rate of growth of $p$-ranks of $p$-class groups in the tower.  For the simplest $\Z_p$-extensions, i.e.~ the cyclotomic ones, he conjectured that $\mu=0$; this was verified for base fields which are abelian over $\Q$ by Ferrero and Washington \cite{FW} but remains an outstanding problem for more general base fields.
Iwasawa initially suspected that his
$\mu$-invariant vanishes for all $\Z_p$-extensions, but later was the first to construct $\Z_p$-extensions with non-zero (indeed arbitrarily large) $\mu$-invariants.  It is natural to ask what other
$p$-adic groups enjoy this property.
Our present work leads us to the following conjecture:

\begin{conjecture}\label{conjecture} Let $\Gamma$ be a uniform pro-$p$ group having a non-trivial fixed-point-free automorphism $\sigma$ of order $m$ co-prime to $p$ (in particular if $m=\ell$ is prime, $\Gamma$ is nilpotent). Then $\Gamma$
has arithmetic realizations with arbitrarily large $\mu$-invariant, i.e.~ for all $n\geq 0$, there exists a number field $\K$ and an extension $\L/\K$ with Galois group isomorphic to $\Gamma$ such that $\mu_{\L/\K} \geq n$. 
\end{conjecture}

Our approach for realizing $\Gamma$ as a Galois group is to make use of the existence of so-called $p$-rational fields.
See below for the definition, but for now let us just say that the critical property of $p$-rational fields is that in terms of certain maximal $p$-extensions with restricted ramification, they behave especially well, almost as well as the base field of rational numbers.
As we will show, Conjecture \ref{conjecture} can be reduced to finding a  $p$-rational field with a fixed-point-free automorphism of order $m$ co-prime to $p$.  These considerations lead us to formulate the following conjecture about $p$-rational fields.

\begin{conjecture}\label{conj2}
Given a prime $p$ and an integer $m \geq 1 $ co-prime to $p$, there exist a totally imaginary field $\K_0$ and a
 degree $m$  cyclic extension  $\K/\K_0$ such that $\K$ is $p$-rational.  
\end{conjecture}

Although we will not need it, we believe $\K_0$ in the conjecture may be taken to be imaginary quadratic; see Conjecture \ref{conjecture-CL} below. Our key result is:

\begin{theo}
Conjecture \ref{conj2} for the pair $(p,m)$ implies Conjecture \ref{conjecture} for any uniform pro-$p$ group $\Gamma$ having a fixed-point-free automorphism of order $m$.
\end{theo}

One knows that if $m$ is an odd divisor of $p-1$, where $p$ is a regular prime, then for any $n\geq 1$, the cyclotomic field $\Q(\zeta_{p^n})$ provides a positive answer to the previous question.  We therefore have

\begin{coro}\label{maintheorem} Assume $p$ is a regular prime and that the uniform group $\Gamma$ has a fixed-point-free automorphism $\sigma$  of odd order $m$ dividing $p-1$. Then Conjecture \ref{conjecture} is true for $\Gamma$.
\end{coro}

\medskip

 This circle of ideas is closely related  to  the recent work of Greenberg \cite{Greenberg} in which he constructs analytic extensions of number fields having a Galois group isomorphic to  an open subgroup of $\Gl_k(\Z_p)$. The idea of studying pro-$p$ towers equipped with a fixed-point-free action of a finite group of order prime to $p$ occurs also in Boston's papers \cite{Boston2} and \cite{Boston3}.

\medskip

Our work raises the following purely group-theoretical question.
 \begin{Question}\label{GTquestion}
  Let $\Gamma$ be a nilpotent uniform pro-$p$ group. Does there exist a uniform nilpotent pro-$p$ group $\Gamma'$
  having a fixed-point-free automorphism of prime order $\ell \neq p$ such that $\Gamma' \twoheadrightarrow \Gamma$~?
 \end{Question}
A positive answer to this question would imply that for all nilpotent uniform pro-$p$ groups $\Gamma$, there exist arithmetic realizations with arbitrarily large $\mu$-invariant.
\medskip

In his  recent work \cite{gras-heuristics}, Gras gave some conjectures about the $p$-adic regulator in a fixed number field $\K$  when $p$ varies. In our context, one obtains:

\begin{theo} Let $\PP$ be an infinite set of prime numbers and let $m$ be an integer co-prime to all $p \in \PP$.
Let $(\Gamma_p)_{p \in \PP}$ be a family of uniform pro-$p$ groups of fixed dimension $d$, all having a fixed-point-free automorphism $\sigma$ of order $m$.
Assuming the  Conjecture of Gras (see Conjecture \ref{conjectureregulator}), there exists a constant $p_0$, such that for all $p \geq p_0$, there exist $\Gamma_p$-extensions of number fields with arbitrarily large $\mu$-invariants.
\end{theo}

In another direction,  a conjecture in the spirit of the heuristics of Cohen-Lenstra concerning the $p$-rationality of the families $\Ff_{\rm G}$ of number fields $K$ Galois over $\Q$, all having Galois group isomorphic to a single finite group ${\rm G}$, seems to be reasonable (see \cite{Pitoun-Varescon}).
When the prime $p \nmid |{\rm G}|$, the philosophy here is that the  density of $p$-rational number fields in $\Ff_{\rm G}$ is positive.
This type of heuristic lends further evidence for conjecture \ref{conjecture}.

\section*{Notations}

Let $\G$ be a finitely generated pro-$p$ group. For two elements $x,y$ of $\G$, we denote by $x^y:=y^{-1}xy$ the conjugate of $x$ by $y$ and by $[x,y]:=x^{-1}y^{-1}xy=x^{-1}x^y$ the commutator of $x$ and $y$. For 
closed subgroups $\H_1, \H_2$ of $\G$, let $[\H_1, \H_2]$ be the \emph{closed} subgroup generated by all commutators $[x_1, x_2]$ with $x_i \in \H_i$. 
Let  $\G^{\ab}:=\G/[\G,\G]$ be the maximal abelian quotient of $\G$, and let  $d(\G):=\dim_{\fq_p}\G^{\ab}$ be its $p$-rank.

\medskip

Denote by  $(\G_n)$ the  $p$-central descending series of $\G$: $$\G_1=\G, \ \G_2=\G^p[\G,\G], \cdots, \ \G_{n+1}=\G_n^p[\G,\G_n]\cdot$$ The sequence  $(\G_n)_n$ forms a base of open neighborhoods of the unit element $e$ of $\G$.

\medskip

If $\K$ is a number field, let  $\A(\K)$ be the $p$-Sylow subgroup of the class group of $\K$. Let $S_p:=\{ \p \subset \O_\K: \ \p | p\}$ be the set of primes
of $\K$ of residue characteristic $p$.
If $S$ is any finite set of places of $\K$, denote by  $\K_S$ the maximal pro-$p$ extension of $\K$ unramified outside $S$ and put $\G_S:=\Gal(\K_S/\K)$ as well as  $\A_S:= \G_S^{\ab}$.

\section{Arithmetic background}

\subsection{Formulas in non-commutative Iwasawa Theory}

\subsubsection{Algebraic tools} \label{sectionuniform}

Two standard  references concerning $p$-adic analytic and in particular, uniform, pro-$p$ groups  are the long article of Lazard \cite{lazard} and the book of Dixon, Du Sautoy, Mann and Segal  \cite{DSMN}.

\medskip

Let $\Gamma$ be an analytic pro-$p$ group: we can think of $\Gamma$ as a closed subgroup of $\Gl_m(\Z_p)$ for a certain integer $m$. The group $\Gamma$ is said \emph{powerful} if $  [\Gamma,\Gamma]\subset \Gamma^p $ ($  [\Gamma,\Gamma] \subset \Gamma^4$ when $p=2$); a powerful pro-$p$ group $\Gamma$ is said  \emph{uniform} if it has no torsion. 

Let us recall two important facts.

\begin{theo}
Every  $p$-adic analytic pro-$p$ group contains an open uniform subgroup. 
\end{theo}

\begin{theo}\label{isomorphismxp}
A powerful pro-$p$ group $\Gamma $ is uniform if and only if for $i\geq 1$,
the map $x\mapsto x^p$ induces an isomorphism between $\Gamma_{i}/\Gamma_{i+1}$ and $ \Gamma_{i+1}/\Gamma_{i+2}$.
\end{theo}

Let us make some remarks.

\begin{rema} Let $\dim(\Gamma)$ be the dimension of  $\Gamma$ as analytic variety. 

1) If $\Gamma$ is uniform then $d_p \Gamma = \dim(\Gamma)= \cd_p(\Gamma)$, where
$\cd_p(\Gamma)$ is the $p$-cohomological dimension of $\Gamma$.

2) \cite[Corollary 1.8]{Klopsch-Snopce} Suppose $\Gamma$ is a torsion-free $p$-adic analytic group, $p \geq \dim(\Gamma) $ and  $d(\Gamma)= \dim(\Gamma)$. Then the group $\Gamma$ is uniform.

3) For $p \geq 2$, the pro-$p$-group $\mathbb{1}_n+\M_n(\Z_p)$ is uniform, where $\M_n(\Z_p)$ is the the set of   $n\times n$ matrices with coefficients in  $\Z_p$.
\end{rema}

\medskip 

 Now let us fix a uniform pro-$p$ group  $\Gamma$ of dimension $d$. 
Recall that $(\Gamma_n)$ is the $p$-descending central series of $\Gamma$.  By Theorem \ref{isomorphismxp}, one has
$[\Gamma:\Gamma_n]=p^{dn}$, for all $n$.

\medskip

Let $\displaystyle{\Z_p[[\Gamma]]:=\lim_{\stackrel{\longleftarrow}{\U \lhd_O \Gamma }} \Z_p[\Gamma/\U]}$ be the complete Iwasawa algebra, where $\U$ runs along the open normal subgroups of $\Gamma$. Put $\Omega:=\fq_p[[\Gamma]] = \Z_p[[\Gamma]]/p$. The rings $\Omega$ and $\Z_p[[\Gamma]]$ are local, noetherien and without zero divisor \cite[\S 7.4]{DSMN}:  each of them    has a fractional skew field. Denote by   $Q(\Omega)$ the fractional skew field of $\Omega$. 
If $\X$ is a finitely generated $\Omega$-module, the \emph{rank} $\rg_\Omega(\X)$ of $\X$ is the  $Q(\Omega)$-dimension of $\X\otimes_{\Omega} Q(\Omega)$.
For more details, we refer the reader to Howson \cite{Ho} and Venjakob \cite{venjakob}.

\

\begin{defi} Let $\X$ be a finitely generated $\Z_p[[\Gamma]]$-module. Put $$r(\X) = \rg_\Omega \X[p] \, \, \, {\rm and \ } \, \mu(\X)= \sum_{i\geq 0}
\rg_\Omega \X[p^{i+1}]/\X[p^{i}].$$
\end{defi}

\begin{rema} One has $\mu(\X) \geq r(\X)$ and, 
 $r(\X) =0$ if and only if $\mu(\X)=0$.
\end{rema}

There is a large and growing literature on the study 
of $\Z_p[[\Gamma]]$-modules in the context of Iwasawa theory. We recall
a result of Perbet \cite{Perbet2}, where, by making use of the work of Harris \cite{Harris}, Venjakob \cite{venjakob} and  Coates-Schneider-Sujatha \cite{CSS},  among others, he manages to estimate the size of the coinvariants $(\X_{\Gamma_n})_n$  of $\X$.  Recall that $\X_{\Gamma_n}$ is the largest quotient of $\X$ on which $\Gamma_n$, the $n$th element of the $p$-central series, acts trivially.

\begin{theo}[Perbet, \cite{Perbet2}] Suppose that   $\X$ is a torsion $\Z_p[[\Gamma]]$-module where $\Gamma$ is a uniform pro-$p$ group of dimension $d$.   Then for $n \gg 0$:
 $$\dim_{\fq_p} \X_{\Gamma_n} =r(\X) p^{dn} + \bigO{p^{n(d-1)}},$$
 and
 $$\# (\X_{\Gamma_n}/p^n) = p^{\mu(\X) p^{dn} + \bigO{np^{n(d-1)}}}.$$
\end{theo}

We now turn to applying these formulas in the arithmetic context. 

\subsubsection{Arithmetic}\label{sectionarithmetic}

Let $\L/\K$  be a Galois extension of number fields with Galois group $\Gamma$, where $\Gamma$ is a uniform pro-$p$ group.
We assume that \emph{the set of primes of $\K$ ramified in $\L/\K$ is finite}.
Let $(\Gamma_n)_n$ be the $p$-central descending series of $\Gamma$ and put $\K_n:=\L^{\Gamma_n}$ for the corresponding tower of fixed fields. 

\

Now let  $\X$ be the projective limit along $\L/\K$ of the $p$-class group $\A(\K_n)$ of the fields $\K_n$. 
Then, $\X$ is a finitely generated torsion $\Z_p[[\Gamma]]$-module.  For the remainder of this work, $\X$ will denote this module built up from the $p$-class groups of the intermediate number fields in $\L/\K$.  In particular, we put $\mu_{\L/\K}=\mu(\X)$ and $r_{\L/\K}=r(\X)$.  For this module, by classical descent, Perbet proved:

\begin{theo}[Perbet \cite{Perbet2}] For $n \gg 0$, one has:  $$\log |\A(\K_n)/p^n| = \mu_{\L/\K} p^{dn} \log p   + \bigO{np^{d(n-1)}}$$
 and $$d_p\A(\K_n)= r_{\L/\K} p^{dn} +\bigO{p^{n(d-1)}}.$$ 
\end{theo}

One then obtains immediately the following corollary:

\begin{coro}\label{critereminorationmu}

$(i)$ 
 Along the extension $\L/\K$, the $p$-rank of $\A(\K_n)$ 
 grows linearly with respect to the degree $[\K_n:\K]$
 if and only if $\mu_{\L/\K} \neq 0$.
 
 $(ii)$ If there exists a constant $\alpha$ such that for all $n\gg 0$,  $d_p \A(\K_n) \geq \alpha p^{dn}$, i.e. $d_p \A(\K_n) \geq \alpha[\K_n:\K]$, then $\mu_{\L/\K} \geq \alpha$. 
\end{coro}

At this point, we should recall some standard facts from commutative Iawasawa theory.  First, for the cyclotomic $\Z_p$-extension, Iwasawa conjectured that 
the $\mu$-invariant is 0 for all base fields, and this has been shown to be the case when the based field is abelian over $\Q$.  When the base field $\K$ contains a primitive $p$th root of unity, the reflection principle allows one to give some
estimates on the $\mu$-invariant  \cite{Perbet2}.

\subsection{On the $p$-rational number fields}\label{section-p-rational}

A number field  $\K$ is called \emph{$p$-rational} if the Galois group $\G_{S_p}$ of the maximal pro-$p$-extension of $\K$ unramified outside $p$ is pro-$p$ free. From the extensive literature on  $p$-rational fields, we refer in particular to Jaulent-Nguyen Quang Do \cite{JN}, Movahhedi-Nguyen Quang Do  \cite{MN}, Movahhedi \cite{M}, and Gras-Jaulent\cite{GJ}.
A good general reference  is the book of Gras \cite{gras}.

\medskip

\begin{defi}
If $\G$ is a pro-$p$ group, let us denote by $\To(\G)$ the torsion of $\G^{\ab}$. 
When $\G=\G_S$, put $\To_S:=\To(\G_S)$ ; when $S=S_p$, put $\To_p:=\To_{S_p}$.
\end{defi}

A standard argument in pro-$p$ group theory shows the following:
\begin{prop}
A pro-$p$ group $\G$ is free if and only if $\To(\G)$ and  $H^2(\G,\Q_p/\Z_p)$ are trivial.
\end{prop}

\begin{proof}
Indeed, the exact sequence $\displaystyle{0 \longrightarrow \Z/p\Z \longrightarrow \Q/\Z \longrightarrow \Q/\Z \longrightarrow 0}$ gives
the sequence $$0 \longrightarrow H_2(\G,\Z)/p  \longrightarrow H_2(\G,\fq_p)  \longrightarrow H_1(\G,\Z)[p]\longrightarrow 0,$$
and to conclude,  recall that $H_1(\G,\Z)\simeq \G^{ab}$ and $H_2(\G,\Z)\simeq H^2(\G,\Q/\Z)^*$.\end{proof}

Remark that  if $\G$ is pro-$p$ free then $\G^{\ab} \simeq \Z_p^d$, where $d$ is the $p$-rank of $\G$.
This observation shows in particular that if  the group $\G$ corresponds to the Galois group of a pro-$p$ extension of number fields, then necessarily this extension is wildly ramified.

\

We now explain how the Schur multiplier of $\G_S$ is related to the Leopoldt Conjecture.

\begin{prop}
Let $S$ be a set of places of $\K$, and for $v \in S$, let $\U_v$ be the pro-$p$ completion of the local units  at $v$. If the $\Z_p$-rank of the diagonal image of $\O_\K^\times$ in $\prod_{v\in S}\U_v$ is maximal, namely $r_1+r_2-1$, then $H^2(\G_S,\Q_p/\Z_p)=0$. 
Moreover if $S_p \subseteq S$, these conditions are equivalent. In particular,  assuming the Leopoldt Conjecture for $\K$ at $p$,   $\G_{S_p}$ is free if and only if $\To_p$ is trivial.
\end{prop}

\begin{proof}
For the case of restricted ramification, see  \cite{maire}. For the general case, see \cite[Corollaire 1.5]{Nguyen}. 
\end{proof}

\begin{rema}
Some examples of free $\G_S$ are given in \S \ref{splitting-section} below. 
\end{rema}

\begin{prop}[A `numerical' $p$-rationality criterion]
Let
 $\A_{\mm}$ be the $p$-Sylow subgroup of the ray class  group of modulus $\displaystyle{\mm= \prod_{\p|p} \p^{a_\p}}$ of $\K$, where $a_\p=2 e_\p+1$, and $e_\p$ is the absolute index of ramification of $\p$ (in $\K/\Q$).  Assume that $\K$ verifies the Leopoldt Conjecture at $p$. 
Then $\K$ is $p$-rational if and only if $d_p \A_{\mm}=r_2+1$.
\end{prop}

\begin{proof}
First, as one assumes Leopoldt Conjecture for $\K$ at $p$, then the $\Z_p$-rank of $\A_{S_p}=\G_{S_p}^\ab$ is exactly $r_2+1$. Hence $\G_{S_p}$ is free if and only if $\To_p=\{1\}$, which is equivalent to $d_p \A_{S_p}=r_2+1$. Now, by Hensel's lemma,   every unit $\varepsilon \equiv 1 (\mod \pi_\p^{a_\p})$ is  a $p$-power in $\K_\p^\times$ and so $d_p  \A_{\mm}=d_p \A_{S_p}$.
\end{proof}

\begin{exem} \label{exemple-p=37}
Take $\K=\Q(\zeta_7)$ and $p=37$. Then, $p$ splits totally in $\K/\Q$ and $a_\p=3$. A simple computation gives $\d_p \A_{S_p}=4=r_2+1$, so $\K$ is $37$-rational.
\end{exem}

\begin{exem}\label{exemple-p=2}
One can check easily that $\Q(\zeta_7)$ is not $2$-rational but $\Q(\zeta_{13})$ is $2$-rational.
\end{exem}

 Here is a very well-known case of the situation.

 \begin{prop}[A `theoretical' $p$-rationality criterion] \label{p-rational}
 Suppose that the number field $\K$ contains a primitive $p$th root of unity. 
 Then $\K$ is $p$-rational if and only if
 there exists exactly one prime of $\K$ above $p$ and the $p$-class group of $\K$ (in the narrow sense if $p=2$) is generated 
 by  the unique prime dividing $p$.
\end{prop}

\begin{proof}
 See for example Theorem 3.5 of \cite{gras}.
 \end{proof}

\begin{rema}
In particular, when $p$ is regular, $\Q(\zeta_{p^n})$ is $p$-rational for all $n\geq 1$.
\end{rema}

Now we will give some more precise statements by considering some  Galois action.

Let $\K/\K_0$ be a Galois extension of degree $m$, with $p \nmid m$. Put  $\Delta=\Gal(\K/\K_0)$.
Let $r=r_1(\K_0)+r_2(\K_0)$ be the number of archimedean places of $\K_0$.
Let $S$ be a finite set of places of $\K_0$.
The arithmetic objects that will use have a structure  of $\fq_p[\Delta]$-modules. Then for a such module $M$,
one notes by  $\chi(M)$ its character. Let $\omega$ be the Teichm\"uller character, let  $\1$ be the trivial character and let $\chi_{\mathrm{reg}}$ be the regular character of $\Delta$.

\begin{prop}\label{action}
Suppose the field $\K$ is $p$-rational and that the real infinite places of $\K_0$ stay real in $\K$ (this is always the case when $m$ is odd). Then $$\chi(\A_{S_p})=r_2(\K_0) \chi_{\mathrm{reg}} + \1\cdot$$
\end{prop}

\begin{proof} It is well-known. The character of the $\Delta$-module   $\prod_{v\in S_p}\big(\U_v/\mu_{p^\infty}(\K_v)\big)$ is equal to $[\K_0:\Q]\chi_{\mathrm{reg}}$ and by Dirichlet's Unit Theorem, the character of $\O_\K^\times/\mu(\K)$ is equal to $\displaystyle{\big(r_1(\K_0)+r_2(\K_0)\big) \chi_{\mathrm{reg}} - \1}$. Then, as $\K$ is $p$-rational, 
$$\chi(\A_{S_p})= [\K_0:\Q]\chi_{\mathrm{reg}} - \Big(\big(r_1(\K_0)+r_2(\K_0)\big) \chi_{\mathrm{reg}} - \1\Big)=r_2(\K_0) \chi_{\mathrm{reg}} + \1\cdot$$
\end{proof}

\subsection{Genus Theory}

The literature on Genus Theory is rich. The book of Gras  \cite[Chapter IV, \S 4]{gras} is a good source for its modern aspects.  All we will need in this work is the following simplified version of the main result.

 \begin{theo}[Genus Theory] \label{theoriegenre}
  Let $\F/\K$ be a Galois degree $p$ extension of number fields $\L/\K$. Let $S$ be the set of places of  $\K$ ramified in 
$\F/\K$ (including the infinite places). Then $$d_p \A(\L) \geq |S| - 1 + d_p \O_K^\times.$$ 
  \end{theo}

We will also need the following elementary fact, whose proof we leave to the reader.

\begin{prop}\label{ramsurQ}
Let $\K=\Q$ and let  $S=\{\ell_1,\cdots, \ell_t \}$ be a finite set of prime numbers such that, for $i=1,\cdots, t$,   $\ell_i \equiv 1 (\mod p)$, where $p$ is an odd prime.
Then there exists a cyclic extension   $L/\Q$ of degree $p$,   exactly and totally ramified at the primes in $S$ (but unramified outside $S$).
\end{prop}
  
 
  \section{Background on automorphisms of  pro-$p$ groups}
 For this section, our main reference is the book of  Ribes and Zalesskii \cite[Chapter 2 and Chapter 4]{RZ}.
  If  $\Gamma$ is a finitely generated pro-$p$ group, denote by $\Aut(\Gamma)$ the group  of continuous automorphisms of  $\Gamma$.
  Recall  that the kernel of $g_\Gamma : \Aut(\Gamma) \rightarrow \Aut(\Gamma/\Gamma_2)$ is a pro-$p$ group and that  $\Aut(\Gamma/\Gamma_2) \simeq \Gl_d(\fq_p)$, where $d$ is the $p$-rank of  $\Gamma$.

  \subsection{Fixed points}

  \begin{defi}
  Let   $\Gamma$ be a finitely generated pro-$p$ group and  let $\sigma \in \Aut(\Gamma)$. Put $\Fix(\Gamma,\sigma):=\{x \in \Gamma : \sigma(x)=x\}$.
  \end{defi}

\begin{rema} \label{primetop}
(i) Obviously,   $\Fix(\Gamma,\sigma)$ is a closed subgroup of $\Gamma$.

(ii) For integers $n$,  $\Fix(\Gamma,\sigma) \subset \Fix(\Gamma,\sigma^n)$ with equality when $n$ is co-prime to the order of $\sigma$.  \end{rema}

  \begin{defi}
  For $\sigma\in \Aut(\Gamma)$, $\sigma \neq e$, one says that  $\act{\Gamma}$ is fixed-point-free if $\Fix(\Gamma,\sigma)=\{e\}$.
   More generally, if $\Delta$ is a subgroup of $\Aut(\Gamma)$, one says that the action $\ACT{\Gamma}{\Delta}$ of $\Delta$ on $\Gamma$ is fixed-point-free if and only if 
   $$\bigcup_{\stackrel{\sigma\neq e}{\sigma \in \Delta}} \Fix(\Gamma,\sigma)=\{e\}.$$ 
   In others words, $\Delta$ is fixed-point-free if and only if, for all non-trivial $\sigma\in \Delta$, $\act{\Gamma}$ is fixed-point-free.
    \end{defi}

 \begin{rema}
Clearly if $\ACT{\Gamma}{{\langle \sigma \rangle}}$ is fixed-point-free then $\act{\Gamma}$ is fixed-point-free; it is an equivalence when $\sigma$ is of prime order $\ell$.
 \end{rema}
 
We are  interested in instances of groups with fixed-point-free action that arise in arithmetic contexts.  
 Let us recall the Schur-Zassenhaus Theorem for a profinite group $\G$:
 
  \begin{theo}[Schur-Zassenhaus] \label{Schur-Zassenhaus}
Let $\Gamma$ be a closed normal pro-$p$ subgroup  of a profinite group $\G$. Assume that  the quotient $\Delta:=\G/\Gamma$ is of order co-prime  to $p$.
Then the profinite group $\G$ contains a subgroup $\Delta_0$ isomorphic to $\Delta$. Two such groups are conjugated  by an element of $\Gamma$ and 
$\G=\Gamma \rtimes \Delta_0$. 
\end{theo}

\begin{proof}
 See Theorem 2.3.15 of \cite{RZ} or proposition  1.1 of \cite{GHR}.
  \end{proof}
 
 As first consequence, one has the following:
 
 \begin{prop}\label{fixedpointfreegeneral}
Let $\Gamma$ be a finitely generated pro-$p$ group and let $\Delta \subset \Aut(\Gamma)$ of order $m$, $p\nmid m$.
 If $\ACT{\Gamma}{\Delta}$ is fixed-point-free  then $\ACT{\big(\Gamma/\Gamma_2\big)}{\Delta}$ is fixed-point-free, where we recall that $\Gamma_2$, the 2nd step in the $p$-central series of $\Gamma$, is the Frattini subgroup.
 \end{prop}

\begin{proof}  Let $\sigma \in \Delta$. The group $\langle \sigma \rangle$ acts on $\Gamma$, on  $\Gamma_2$ and on $\Gamma/\Gamma_2$. 
 By the Schur-Zassenhaus Theorem (applied to $\Gamma_2 \rtimes \langle \sigma \rangle$),
the non-abelian cohomology group  $H^1(\langle \sigma \rangle,\Gamma_2)$ is trivial and then the nonabelian cohomology of the exact sequence $1\longrightarrow \Gamma \longrightarrow \Gamma \rtimes \langle \sigma \rangle \longrightarrow \langle \sigma \rangle  \longrightarrow 1$ allows us to obtain: $$H^0(\langle \sigma \rangle ,\Gamma)\twoheadrightarrow H^0(\langle \sigma \rangle,\Gamma/\Gamma_2),$$
which is exactly the assertion of the Proposition. 
  See also   \cite{Boston},  \cite{Greenberg}, \cite{wingberg}.
\end{proof}

A main observation for our paper is the converse of the previous proposition when $\Gamma$ is uniform:
  
 \begin{prop}\label{actionuniforme} Let $\Gamma$ be a uniform pro-$p$ group.
  Let $\sigma \in  \Aut(\Gamma)$  of order $m$ co-prime to $p$.
  Then $\ACT{(\Gamma/\Gamma_2)}{\sigma}$ is fixed-point-free if and only if   $\ACT{\Gamma}{\sigma}$ is fixed-point-free.
 \end{prop}

 \begin{proof}One direction is taken care of by Proposition \ref{fixedpointfreegeneral}. For the other direction, we first note that
  $$\begin{array}{rcl}  \varphi  : \Gamma_{n-1}/\Gamma_n & \rightarrow &\Gamma_n/\Gamma_{n+1} \\
  x & \mapsto  & x^p
 \end{array}$$
 is a $\langle \sigma \rangle$-isomorphism for $n\geq 2$.  We thus obtain a $\langle \sigma \rangle$-isomorphism from $\Gamma/\Gamma_2$ to $\Gamma_{n}/\Gamma_{n+1}$.
 If $\act{(\Gamma/\Gamma_2)}$ is fixed-point-free, then $\act{(\Gamma_n/\Gamma_{n+1})}$ is fixed-point-free for all $n\geq 1$.
Suppose $y \in \Gamma$ satisfies $\sigma(y)=y$. As the action of $\Gamma/\Gamma_2$ is fixed-point-free, we have $ y (\mod \Gamma_2) \in  \Gamma/\Gamma_2$ is trivial, so $y\in \Gamma_2$. Continuing in this way, we in fact conclude that 
 $\displaystyle{y\in \bigcap_n \Gamma_n=\{e\}}$.
 \end{proof}

  \begin{coro}
  Let $\Gamma$ be a uniform pro-$p$ group and let $\sigma \in \Aut(\Gamma)$ of order  $m$ co-prime to $p$. Denote by $\chi$ the character of the semi-simple action of $\langle \sigma \rangle$ on $\Gamma/\Gamma_2$.
  Then $\act{\Gamma}$ is fixed-point-free if and only, $\langle \chi, \1 \rangle =0$.
  \end{coro}
  
  \begin{proof}
  Indeed, one has seen (Proposition \ref{actionuniforme}) that  $\act{\Gamma}$ is fixed-point-free if and only $\act{\big(\Gamma/\Gamma_2\big)}$ is fixed-point-free, which is equivalent to $\langle \chi, \1 \rangle =0$.
  \end{proof}

  \begin{rema} Suppose that $\Gamma$ is uniform of dimension $d$.
  The restriction of $\sigma \in \Aut(\Gamma)$ to $\Gamma/\Gamma_2 \simeq \fq_p^d$ is an element of $\Gl_d(\fq_p)$. Denote by $P_\sigma\in \fq_p[X] $ its characteristic polynomial. Then 
  the action $\act{\Gamma}$ is fixed-point-free if and only if $P_\sigma(1) \neq 0$.
  \end{rema}

  \begin{rema}
When $\Gamma$ is uniform and $\Delta \subset  \Aut(\Gamma)$ is of order $m$ co-prime to $p$, testing that the action of $\Delta$ is fixed-point-free on $\Gamma$ is equivalent to testing this condition on the quotient ${\rm M}:=\Gamma/\Gamma_2$. Let $\chi$ be the character resulting from the action of $\Delta$ on ${\rm M}$.  
 Then $\ACT{\Gamma}{\Delta}$  is fixed point-free on ${\rm M}$ if and only if, for all $e \neq \sigma \in \Delta$,  $\Res_{|\langle \sigma \rangle}(\chi)$ does not contain the trivial character, where here $\Res_{|\langle \sigma \rangle}$ is the restriction to $\langle \sigma \rangle$.  By Frobenius Reciprocity, this condition is equivalent to $\langle \chi, \Ind_{\langle \sigma \rangle }^\Delta \1\rangle =0$, 
where $\Ind_{\langle \sigma \rangle }^\Delta$ is the induction from $\langle \sigma \rangle$ to $\Delta$.
\end{rema}

  We need also of the following proposition which will be crucial for our main result.
  
  \begin{prop}\label{conjugueauto}
  Let  $\Gamma$ be a finitely generated pro-$p$ group.
  Let $\sigma$ and $\tau$ be two elements in $\Aut(\Gamma)$ of order $m$ co-prime to $p$. If $\sigma \Gamma_2 = \tau  \Gamma_2$ in $\Gamma/\Gamma_2$, then 
  $\displaystyle{\Gamma^{\circlearrowleft^\sigma}}$ is fixed-point-free if and only if, 
  $\displaystyle{\Gamma^{\circlearrowleft^\tau}}$ is fixed-point-free.
 More precisely,  $\Fix(\Gamma,\sigma) = g\cdot \Fix(\Gamma,\tau)$ for a certain element  $g \in \ker\big(\Aut(\Gamma)\rightarrow \Aut(\Gamma/\Gamma_2)\big)$.
 \end{prop}

 One needs  the following lemma:
 
 \begin{lemm}\label{autoexterieur}  Let $\Gamma$ be a finitely generated pro-$p$ group.
  Let  $\sigma$ and $\tau$ be two elements of  $\Aut(\Gamma)$ of order $m$, $p \nmid m$, satisfying $\sigma \Gamma_2= \tau\Gamma_2$.
Then there exists $g \in \Aut(\Gamma)$ such that
  $\tau=\sigma^g$.
 \end{lemm}

 \begin{proof} This is to be found in Lemma 3.1 of \cite{HR}.  Since 
  $\sigma$ and  $\tau$ coincide as elements of $\Aut(\Gamma/\Gamma_2)$, there exists $\gamma$ in the pro-$p$ group $\ker\big(\Aut(\Gamma)\rightarrow \Aut(\Gamma/\Gamma_2)\big)$ such that
 $\sigma=\gamma \tau$. 
Consider the group $\langle\tau, \gamma\rangle :\langle \gamma, \gamma^{\tau}, \cdots, \gamma^{\tau^{m-1}}\rangle\rtimes \langle\tau\rangle$. 
Since $\langle \gamma, \gamma^{\tau}, \cdots, \gamma^{\tau^{m-1}}\rangle \subset
\ker\big(\Aut(\Gamma)\rightarrow \Aut(\Gamma/\Gamma_2)\big) $, $\langle\tau,\gamma\rangle$ 
is a semi-direct product
of a pro-$p$-group and a group of order $m$. As $\tau$ and $\sigma$ are both in $\langle\tau,\gamma\rangle$, the subgroups $\langle\tau\rangle$ and $\langle\sigma\rangle$ are conjugate to each other (by the Schur-Zassenahus Theorem \ref{Schur-Zassenhaus}):
 there exists $g\in \langle \gamma, \gamma^{\tau}, \cdots, \gamma^{\tau^{m-1}}\rangle$
such that $\tau^k=\sigma^g$ for a certain integer $k$, $(k,m)=1$, since $\sigma$ and  $\tau$ have the same order.  Moreover, in  $\Gamma/\Gamma_2$, $\overline{\tau}^k=\overline{\sigma^g}=\overline{\sigma}=\overline{\tau}$,
 and so we can take $k=1$.
  \end{proof}

\begin{proof}[Proof of Proposition \ref{conjugueauto}]
  By the previous lemma, $\tau=\sigma^g$, for some $g$ in $\Aut(\Gamma)$.  Moreover, $\sigma\Gamma_2=\tau\Gamma_2$ 
  implies that $g$ induces the trivial automorphism of $\Gamma/\Gamma_2$.  
  We have that $y$ is a fixed point of $\tau$ if and only if $g(y)$ is a fixed point of $\sigma$.
 \end{proof}

  \subsection{Lifts}
  
Given a uniform pro-$p$ group $\Gamma$ equipped with an automorphism $\sigma$ of order
$m$ prime to $p$, the central question of this subsection is to realize 
  $\Gamma \rtimes \langle \sigma \rangle$ as a Galois extension over a number field.
  
   
  \begin{prop}\label{relevementlibre} Let  $\F$ be a free pro-$p$-group on $d$ generators, and let
  $g_\F$ be the natural map $\Aut(\F) \to \Aut(\F/\F_2)$.
 Consider a subgroup $\Delta \subset \Aut(\F/\F^p[\F,\F])$ of order $m$ co-prime to $p$. Then  there exists a subgroup $\Delta_0 \subset \Aut(\F)$ isomorphic to $\Delta$ such that $g_\F(\Delta_0)=\Delta$. Moreover, any two such subgroups are conjugated  by an element
 $g \in \ker\big(\Aut(\F) \rightarrow \Aut(\F/\F_2)\big)$.
\end{prop}

 \begin{proof} First of all the  natural map $g_\F: \Aut(\F) \rightarrow \Aut(\F/\F_2)$ is onto (Proposition 4.5.4 of \cite{RZ}).
 Put $\widetilde{\Delta}= g_\F^{-1}(\Delta)$ and recall that $\ker(g_\F)$ is a pro-$p$ group. Then, $\Delta \simeq \widetilde{\Delta}/\ker(g_F)$ which has order co-prime to $p$. By the Schur-Zassenhaus Theorem \ref{Schur-Zassenhaus}, there exists a subgroup $\Delta_0  \subset \Aut(\F)$, such that $(\Delta_0 \ker g_\F)/\ker(g_\F) \simeq \Delta$. Moreover, two such subgroups are conjugate to each other.
\end{proof}

In fact, one needs a little bit more. The following proposition can be found in a recent paper of Greenberg \cite{Greenberg} and partially in an unpublished paper of Wingberg \cite{wingberg}.

\begin{prop}[Greenberg, \cite{Greenberg}, Proposition 2.3.1]\label{wingberg1}
 Let $\G=\F \rtimes \Delta $  be a profinite group where $\F$ is free pro-$p$ on $d'$ generators and where $\Delta$ is a finite group  of order $m$ co-prime to $p$. Let $\Gamma$ be a finitely generated pro-$p$ group on $d$ generators, with $d'\geq d$. Suppose that there exists  $\Delta' \subset \Aut(\Gamma)$, with $\Delta'\simeq \Delta$, such that the module $\Delta'_{|_{\Gamma/\Gamma_2}}$ is isomorphic to a submodule of $\Delta_{|_{\F/\F_2}}$. Then there exists  a normal subgroup $\N$ of $\F$, stable under $\Delta$, such that $$\G \twoheadrightarrow \Gamma \rtimes \Delta \simeq \Gamma \rtimes \Delta'\cdot$$ 
\end{prop}

Since this result is essential for our construction, we include a proof.

 \begin{lemm}[Wingberg, \cite{wingberg}, lemma 1.3]\label{wingberg0}
 Let  $\F$ be a free pro-$p$-group on $d$ generators and let $\Gamma$ be a pro-$p$ groups generated by $d$ generators.
  Let $\varphi$ be a morphism on pro-$p$ groups
  $\varphi : \F \twoheadrightarrow \Gamma$. Assume that there exists a finite group $\Delta \subset \Aut(\Gamma) $ of order $m$ co-prime to $p$. Then the action of $\Delta$ lift to  
  $\F$ such that  $\varphi$ becomes a $\Delta$-morphism. 
  \end{lemm}
 
 \begin{proof}[Proof following \cite{wingberg}]
 For a finitely generated pro-$p$ group $\N$, denote by $g_\N$ the natural map  $ g_\N: \Aut(\N) \rightarrow \Aut(\N/\N_2)$. Recall that $\ker(g_\N)$ is a pro-$p$ group.
 
 Let $\displaystyle{1 \longrightarrow \R \longrightarrow \F \longrightarrow \Gamma \longrightarrow 1}$ be a minimal presentation of $\Gamma $. Denote by $\Aut_\R(\F):=\{\sigma \in \Aut(\F):  \sigma(\R) \subset (\R)\}$.
  The  natural morphism $f : \Aut_\R(\F) \rightarrow \Aut(\Gamma)$ is onto (see \cite[Proposition 4.5.4]{RZ}).
  Put $\widetilde{\Delta}:= f^{-1}(\Delta) \subset \Aut_\R(\F)$. Then $f(\widetilde{\Delta})=\Delta$.
   Now the isomorphism between $\F/\F_2$ and $\Gamma/\Gamma_2$ induces an isomorphism $f'$ between $\Aut(\F/\F_2)$ and $\Aut(\Gamma/\Gamma_2)$.
Hence  on $\Aut_\R(\F) \subset \Aut(\F)$, one has: $g_\Gamma\circ f = f'\circ g_\F$. In particular $\ker\big(\Aut_\R(\F) \rightarrow \Aut(\Gamma/\Gamma_2)\big)$ is a pro-$p$ group. Now let $\widetilde{f} : \widetilde{\Delta} \rightarrow \Delta$. Then: $(i)$ $\ker(\widetilde{f})$ is a  pro-$p$ group and $(ii)$  $\widetilde{\Delta}/\ker(\widetilde{f}) \simeq \Delta$ which has order co-prime to $p$. By the Schur-Zassenhaus Theorem \ref{Schur-Zassenhaus}, there exists $\Delta_0 \subset \Aut_\R(\F)$ such that $\Delta_0 \cap \ker(\widetilde{f})$, {\it i.e.} $f(\Delta_0)=\Delta$ and we are done.
 \end{proof}

\begin{proof}[Proof of Proposition \ref{wingberg1}] As $d'\geq d$, let $\varphi$ be a surjective morphism $\F \twoheadrightarrow \Gamma$. Put $\R=\ker(\varphi)$.
As $p\nmid m$, the action of $\Delta$ on $\F/\F_2$ is semi-simple. Let us complete the $\fq_p[\Delta]$-module $\Gamma/\Gamma_2$ with a submodule $\MM$ such that $\Gamma/\Gamma_2 \oplus \MM \simeq \F/\F_2$ as   $\Delta$-module.
Let $\Gamma'$ be the pro-$p$ group $\Gamma'=\Gamma \times \Gamma_0$, where $\Gamma_0 \simeq \left(\Z_p/p\Z_p\right)^{d'-d}$ is generated by an $\fq_p$-basis of $\MM$.
By Lemma \ref{wingberg0}, there exists $\Delta_0 \subset \Aut_\R(\F)$ isomorphic to $\Delta$, such that the morphism $\varphi : \F \twoheadrightarrow \Gamma'$  is  a $\Delta_0$-morphism.
By Proposition \ref{relevementlibre}, there exists $g\in \Aut(\Gamma)$ such that $\Delta_0=\Delta^g$.
We note that $\Delta_0 \subset \Aut_\R(\F) $ is equivalent to $\Delta \subset \Aut_{g(\R)}(\F)$.
Then we take $\N=\langle g(\R), g(\MM) \rangle $ and observe  that $\F/\N$ is $\Delta$-isomorphic to~$\Gamma$.
\end{proof}

  \subsection{Frobenius groups}
We now review a group-theoretic notion that we need for our study of the $\mu$-invariant.

 \begin{defi} Let $\G$ be a profinite group. One says that $\G$ is a Frobenius group  if  $\G= \Gamma \rtimes \Delta$, where
 $\Gamma$ is a finitely generated pro-$p$ group, $\Delta$ is  of order  $m$ co-prime to $p$, and such that  the conjugation action  of $\Delta$ on 
  $\Gamma$ is fixed-point-free. 
 \end{defi}

The notion of a Frobenius group is a very restrictive one, as illustrated in the following Theorem:

 \begin{theo}[Ribes-Zalesskii, \cite{RZ}, corollary 4.6.10]\label{RZ}
 Let $\G=\Gamma \rtimes \Delta$ be a Frobenius profinite group. 
 Then the subgroup $\Gamma$ of $\G$ is  nilpotent.  Moreover if $2 \mid  |\Delta |$, $\Gamma$ is abelian, if $3\mid |\Delta|$, $\Gamma$ is  nilpotent of class at most $2$, and more generaly 
  of class at most $\displaystyle{\frac{(\ell-1)^{2^{\ell-1}-1}-1}{\ell-2}}$ if the prime number  $\ell$ divides $|\Delta|$.
 \end{theo}

  \begin{prop}
  Let $\G=\Gamma \rtimes \Delta$, where $\Gamma$ is a uniform pro-$p$ group and such that $p\nmid |\Delta|$. Then $\G$ is a Frobenius group  if and only if the action of $\Delta$ is fixed-point-free on $\Gamma/\Gamma_2$.
  \end{prop}
  
  \begin{proof}
  It is a consequence of  Proposition \ref{actionuniforme}.
  \end{proof}

\section{Proof of the main result}

 Let us recall the motivating question of this paper.  Given a uniform group~$\Gamma$ of dimension $d$, equipped with a fixed-point-free automorphism of finite order co-prime to $p$, can one realize an arithmetic context for $\Gamma$ as Galois group with arbitrarily large associated $\mu$-invariant?
 
 \subsection{The principle}
  
Here we develop our strategy. 
Given a uniform pro-$p$ group~$\Gamma$,
the key task is to produce an extension $\L$ of a number field $\K$ 
with Galois group isomorphic to $\Gamma$ such that 
\begin{enumerate}
\item[$(i)$] there are only finitely many primes that are ramified in $\L/\K$;
\item[$(ii)$]  there exist infinitely many primes of $\K$
that split completely in $\L/\K$. 
\end{enumerate}

To produce such a situation, we realize $\Gamma$, along with its fixed-point-free automorphism, inside the maximal pro-$p$ extension $\K_{S_p}/\K$ of a $p$-rational field $\K$; in particular the condition $(i)$ will be automatically satisfied in our situation.

\begin{prop}\label{principle}
Suppose $\Gamma$ is a uniform pro-$p$ group and $\L$ is a Galois extension of a number field $\K$ such that $\Gal(\L/\K)$ is isomorphic to $\Gamma$ and the following conditions hold:
\begin{enumerate}
\item[$(i)$] there are only finitely many primes that are ramified in $\L/\K$;
\item[$(ii)$]  there exist infinitely many primes of $\K$
that split completely in $\L/\K$. 
\end{enumerate}
Then there exist $\Gamma$-extensions of number fields with arbitrarily large associated 
$\mu$-invariant.  
\end{prop}

 \begin{proof}
 Since only finitely many primes of $\Q$ ramify in $\K$, a field of finite degree over $\Q$, the set $S_0$ of primes of $\K$ which split completely in $\L$ and are neither divisors of $p$ nor ramified over $\Q$ is infinite.  Fix a finite subset $S$ of $S_0$ of any prescribed order.  Let $\ell_1, \ldots, \ell_n$
be the residue characteristics of the primes in $S$.
 Let $\F/\Q$ a cyclic degree $p$ extension, described in Proposition \ref{ramsurQ}, which
  is ramified at each of these $\ell_i$. Then $\L\F/\K\F$ is Galois extension of group isomorphic to $\Gamma$. Put $\K_n=\L^{\Gamma_n}$ and $\F_n:=\F\K_n$.
  Then by applying Theorem \ref{theoriegenre} in $\F_n/\K_n$, one obtains
  $$d_p \A(\F_n) \geq [\F_n:\F]\left( |S|-r_2(\F)\right).$$
  One conclude with Corollary \ref{critereminorationmu}, thanks to the fact that $S$ can be arbitrarily large.

 \end{proof}
 
\subsection{The case $\Gamma = \Z_p$} \label{sectioncommutatif}

We now review the strategy of \cite{Iwasawa} (see also \cite[\S 4.5]{Serre}) for finding arithmetic situations with large $\mu$-invariant.

Let $\K/\Q$ be an imaginary quadratic field.  Let us denote by $\sigma$ the generator of 
$\Gal(\K/\Q)$. Suppose $p$ is a rational prime which splits in $\O_\K$ into two distinct primes $\p_1$ and $\p_2$.  Let us suppose further that $p$ does not divide the class number of $\K$.  Then for $i=1,2$, the maximal pro-$p$ extension $\K_{\{\p_i\}}$ of $\K$ unramified outside $\p_i$ has Galois group $\Gamma_i$ isomorphic to $\Z_p$.  The automorphism $\sigma$ permutes the fields  $\K_{\{\p_i\}}$.  Thus inside the compositum of these two $\Z_p$-extensions, if we denote by  $\langle e_i\rangle$
the subgroup fixing the field $\K_{\{\p_i\}}$, then the subfield $\L$ fixed by $\langle e_1+e_2 \rangle $ is Galois over $\Q$ with Galois group isomorphic to $\Z_p$, and the action $\sigma(e_1)= e_2\equiv -e_1 (\mod \langle e_1+e_2 \rangle )$ is dihedral.

\begin{coro}\label{quadratiquecase}
Under the preceding conditions, $\Gal(\L/\Q) =  \Gal(\L/\K) \rtimes \langle\sigma\rangle  \simeq \Z_p \rtimes \Z/2\Z$, where the action of 
$\sigma$ satisfies $x^\sigma=x^{-1}$, $x$ denoting a generator of $\Gal(\L/\K)$.
Thus, all primes $\ell$ of $\Q$ which remain inert in $\K/\Q$ subsequently split completely in $\L/\K$.
\end{coro}

\begin{proof}
The first part of the corollary follows from the observations preceding it.  For the second part, remark that the non-trivial cyclic subgroups of $\Gal(\L/\Q)$ are of the form $\langle x^{p^k} \rangle \simeq \Z_p$, $k\in \Nat$,  or of the form  $\langle \sigma x^{p^k} \rangle \simeq \Z/2\Z$.  Hence let  $\ell$ be a prime which is unramified in $\L/\Q$. If $\ell$ is inert in $\K/\Q$, then the Frobenius automorphism  $\sigma_\ell$ of $\ell$ in $\L/\Q$ is of the form $\sigma x^{p^k}$. As $\sigma_\ell$ is of order $2$ (or equivalently, $\langle \sigma_\ell \rangle \cap \Gal(\L/\K)=\{e\}$), the prime $\ell$ splits totally in $\L/\K$. 
\end{proof}

By applying Proposition  \ref{principle} to the extension $\L/\Q$, this construction allows us to produce $\Z_p$-extensions with  $\mu$-invariant as large as desired.

\begin{rema}
We can say a bit more in the example \ref{quadratiquecase}. Let us show that a place of $\K$ above a prime $\ell$ splits totally in $\L/\K$ if and only if $\ell$ is inert in $\K/\Q$. For a prime $\q \nmid p$, denote by  $\sigma_\q$ the Frobenius of $\q$ in the compositum  of the $\Z_p$-extensions of $\K$.
Let $\ell$ a prime that splits in $\K/\Q$; let us write $\ell\O_\K=\ll \ll'$.   If we write $\sigma_{\ll}=ae_1+be_2$, then by conjugation, $\sigma_{\ll'}=be_1+ae_2$.  The key point is that the maximal pro-$p$ extension of $\K$ unramified outside $p$ and totally split at $\ll$ and $\ll'$ is finite, see \cite{gras}. Then $a^2\neq b^2$. Note that when $\ell$ is inert in $\K/\Q$ then $\sigma_\ell=a(e_1+be_2)$.
By reducing the Frobenius modulo $\langle e_1+e_2\rangle$, we note that in the case when $\ell$ is inert,  $\sigma_\ell \equiv 0 \  (\mod \ \langle e_1+e_2\rangle)$ but when $\ell$ splits in $\K/\Q$, $\sigma_{\ll}\equiv (a-b)e_1 \ (\mod  \ \langle e_1+e_2\rangle) \neq 0 \ (\mod  \ \langle e_1+e_2\rangle) $.
\end{rema} 
 
 \begin{rema}
One can generalize the above discussion for $\Gamma=\Z_p^r$, by considering a large CM-extension abelian over $\Q$. See  Cuoco \cite[Theorem 5.2]{cuoco})
 \end{rema}

 \begin{rema}
We observe that the basic principle in the constructions above is that there is a positive density of primes of a field $\K_0$ which are inert in $\K$, and that all of these subsequently split (thanks to the fact that the action $\act{\Gamma}$ is fixed-point-free) in the $\Gamma$-extension. This is the starting point for the general case.
 \end{rema}
 
 \subsection{The general case} \label{cadregeneral}

We will consider a uniform  pro-$p$ group $\Gamma$  of dimension $d$  having a fixed-point-free automorphism 
 $\sigma$ of order $m$ co-prime to $p$.   
 We assume that  $m \geq 3$; indeed for $m=2$, $\Gamma \simeq \Z_p^d$ (by Theorem \ref{RZ} of Ribes and Zaleskii).

\begin{prop}\label{existence}
Suppose $\K_0'$ is a totally imaginary number field admitting 
 a cyclic extension $\K'/\K_0'$ of degree $m \geq 3$ co-prime to $p$ such that
 $\K$ is $p$-rational.
Let $\Gamma$ be a uniform group of dimension $d$ having an automorphism $\tau$ of order $m$. Let
$n$ be an integer such that $[\K_0':\Q] \cdot p^n \geq 2d$, and let $\K_0$, respectively
$\K$ be the $n$th layer of the cyclotomic $\Z_p$-extension of $\K_0'$, respectively $\K'$.
Then there exists an intermediate field $\K \subset \L \subset \K_{S_p}$  such that
$\L$ is Galois over $\K_0$ with Galois group isomorphic to $\Gamma \rtimes \langle \tau \rangle$. In particular, if $\tau$ acts fixed-point-freely on $\Gamma/\Gamma_2$ and if $m=\ell$ is prime, then $\Gal(\L/\K_0)$ is a Frobenius group.
\end{prop}

\begin{proof}
The extension $\K/\K_0$ is cyclic of degree $m$ and the number field $\K$ is $p$-rational: put $\Gal(\K/\K_0)=\langle \sigma \rangle$ and $\F:=\Gal(\K_{S_p}/\K)$. The extension $\K_{S_p}/\K_0$ is a Galois extension of group $\G$ isomorphic to $\F\rtimes \langle \sigma \rangle $.
By Proposition \ref{action}, the character of the action of $\sigma$ on $\F/\F_2$ is $\displaystyle{\frac{[\K_0':\Q]}{2} \cdot p^n  \cdot \chi_{\mathrm{reg}}+\1}$.
Now let  $\displaystyle{\chi(\Gamma/\Gamma_2)=\sum_{\chi \in \Irr(\langle \sigma \rangle)}\lambda_\chi \chi}$ be the character of the action of $\tau$ on $\Gamma/\Gamma_2$. Then $\sum_\chi \chi(1) \lambda_\chi =d$ and, for all $\chi$,  $\lambda_\chi \leq d/\chi(1) \leq d$. In particular since $[\K_0':\Q] \cdot p^n \geq 2 d,$
then necessarily, the  $\langle \tau \rangle$-module $\Gamma/\Gamma_2$ is isomorphic to a submodule of $\F/\F_2$.  
 By Proposition \ref{wingberg1}, there exists a normal subgroup $\N$ of $\F$, stable under $\sigma$, such that $\G \twoheadrightarrow \Gamma \rtimes \langle \sigma \rangle \simeq \Gamma \rtimes \langle \tau \rangle$.
\end{proof}

We now state the key arithmetic proposition we need.

\begin{prop}\label{splitting}
  Let $\L/\K_0$ be a Galois extension of Galois group $\Gamma \rtimes \langle \sigma \rangle$, where $\Gamma$ is a uniform group of dimension $d$ and where $\sigma$ is of order $m$ co-prime to $p$. 
  Suppose that $\act{\Gamma}$ is fixed-point-free. Then  every  place $\p$ which is (totally) inert in $\K/\K_0$ and is not ramified in $\L/\K$
  splits completely in $\L/\K$.
 \end{prop}

 \begin{proof}
  Let $\p$ a prime of $\K$ inert in $\K/\K_0$ which is not ramified in $\L/\K$. Let us fix a prime $\P|\p$ of $\L$ 
  (see $\P$ as a system of coherent primes in $\L/\K$). Denote by $\widetilde{\sigma_\p}$ be the Frobenius of $\p$ in $\K/\K_0$ and let $\sigma_\P \in \Gal(\L/\K_0)$ be
  an element of order $m$ of the decomposition group of $\P$ in $\L/\K_0$ lifting  $\widetilde{\sigma_\p}$.
  Then  $\sigma_{\P}= \sigma^i$ in $\Aut(\Gamma/\Gamma_2)$, for an integer $i$, $(i,m)=1$.
 By Proposition \ref{conjugueauto}, there exists $g\in \ker\big(\Aut(\Gamma)\rightarrow \Aut(\Gamma/\Gamma_2)\big)$ such that $\Fix(\Gamma,\sigma_\P)= g \cdot \Fix(\Gamma, \sigma^i)=\{e\}$, the last equality coming from Remark~\ref{primetop}.
  Let $\widehat{\sigma_\P}:=\Frob_\P(\L/\K)$. As $\p \in \Sigma$, the prime $\P$ is unramified in $\K/\K_0$ and then the decomposition group of $\P$ in $\L/\K_0$ is cyclic: 
  the elements $\sigma_\P$ and $\widehat{\sigma_\P}$ commute or, equivalently, $\widehat{\sigma_{\P}}^{\sigma_\P}=\widehat{\sigma_\P}$. Hence if $\widehat{\sigma_\P}\neq e $,
  the element $\sigma_{\P} \in \Aut(\Gamma)$
  has a fixed-point. Contradiction, and then  $\widehat{\sigma_\P}=e$.
 \end{proof}

\begin{theo}\label{main}
Suppose $\K_0'$ is a totally imaginary number field admitting 
 a cyclic extension $\K'/\K_0'$ of degree $m \geq 3$ co-prime to $p$ such that
 $\K'$ is $p$-rational.
Let $\Gamma$ be a uniform pro-$p$  group having an automorphism $\tau$ of order $m$ with fixed-point-free action. There exists  a finite $p$-extension $\K/\K'$ unramified outside $p$ such that for any given integer $\mu_0$, there exists a cyclic degree $p$ extension $\F/\K$ which admits a $\Gamma$-extension $\L/\F$ satisfying $\mu_{\L/\F}\geq \mu_0$.
\end{theo}

 \begin{proof}
 By Proposition \ref{existence}, once we let $\K_0$ and $\K$ by the $n$th layer of the cyclotomic $\Z_p$-extensions of $\K_0'$ and $\K'$ respectively, where $p^n[\K_0:\Q]\geq 2 \dim(\Gamma)$, we are guaranteed of the existence of a Galois extension $\L/\K_0$ with Galois group isomorphic to $\Gamma \rtimes  \langle \tau \rangle$ and $\K=\L^{\Gamma}$. By Proposition \ref{splitting}, every inert places  in $\K/\K_0$ splits completely in $\L/\K$.
To finish, we apply the construction of Proposition \ref{principle}. 
 \end{proof}

 \begin{rema}
Let us remark that in the construction  of Theorem \ref{main} (in fact of Proposition \ref{principle}),  we start with a $p$-rational field $\K $  but its cyclic degree $p$ extension $\F$ with many ramified primes is not $p$-rational. Indeed as $|S| \rightarrow \infty$ and the extension $\F/\K$ is ramified at every place of $S$, the extension $\F/\K$ is not primitively ramified when $S$ becomes large. And then the field $\K\F$ is not $p$-rational (see \cite[Chapter IV, \S 3]{gras}).
 \end{rema}

 \begin{rema} In his original treatment \cite{Iwasawa}, Iwasawa was able to treat the case $p=2$ alongside odd primes $p$.  The elements of finite order of  $\Aut(\Z_2)$ are of order $2$. Then the Question \ref{GTquestion} is essential  for applying our previous ``co-prime to $p$'' 
 strategy for  $\Z_2$. Indeed, let us consider the uniform pro-$2$ group $\Gamma:=\Z_2^2$ instead of $\Z_2$ by noting that 
 $\Aut(\Z_2^2)$ has a fixed-point-free automorphism $\tau$ of order $3$. By example \ref{exemple-p=2}, one knows that the field  $\K=\Q(\zeta_{13})$ is $2$-rational; the Galois group $G_{S_2}$ is free on $7$ generators. Moreover, $G_{S_2}$ has an automorphism of order $3$ coming from the unique cyclic sub-extension $\K/\K_0$ of degree $3$. The character of this action contains the character of the action of $\tau$ on $\Gamma/\Gamma_2$. Hence by Proposition \ref{existence}, there exists a Galois extension $\L/\K$ with $\Gal(\L/\K) \simeq \Gamma = \Z_2^2$ in which every odd inert place $\p$ in $\K/\K_0$ splits completely in $\L/\K$. In particular every such place splits in every $\Z_2$-quotient of $\Gamma$ and then  the Proposition \ref{principle} apply for $\Z_2$.
\end{rema} 
 
 \section{Complements}
 
 \subsection{A non-commutative example} \label{exemplenoncommutatif}

 The nilpotent uniform groups of dimension $\leq 2$ are all commutative. In dimension $3$ they are parametrized, up to isomorphism,
 by $s\in \Nat$ and represented (see \cite[\S 7 Theorem 7.4]{sanchez-klopsch}) by: $$\Gamma(s) = \langle x,y,z  \ | \ [x,z]=[y,z]=1, [x,y]=z^{p^s}\rangle.$$
 Here the center of $\Gamma(s)$ is the procyclic group $\langle z\rangle$ and one has the sequence:$$1\longrightarrow \Z_p \longrightarrow \Gamma(s) \longrightarrow \Z_p^2 \longrightarrow 1.$$

 \begin{prop} Let $s\in \Nat$ and let $p \equiv 1 (\mod \ 3)$.
  The group of automorphisms of $\Gamma(s)$ contains an element $\sigma$ of order $3$ for which the action is fixed-point-free.
 \end{prop}

 \begin{proof} To simplify the notation, put $\Gamma=\Gamma(s)$.
  First, let us remark that for  $a,b \in \Nat$,  one has $x^ay^b= z^{abp^s}y^bx^a$. Indeed,
 $$x^ay=x^{a-1}z^{p^s}yx=z^{p^s}x^{a-1}yx=z^{ap^s}yx^a,$$
 and the same holds for $xy^b$.

 Let $\zeta$ be a primitive  third root of the unity and let  $\zeta^{(n)} \in \Nat$ be the truncation at level $n$ of the $p$-adic expansion of $\zeta$:
 $\zeta^{(n)} \equiv \zeta  (\mod \ p^n)$.
 Let us consider $\sigma$ defined by: 
 $$\begin{array}{rcl} \sigma : \Gamma &\rightarrow & \Gamma\\
    x& \mapsto & x^\zeta \\
    y & \mapsto & y^\zeta \\
    z & \mapsto & z^{\zeta^2}
   \end{array}$$
   Then $\sigma \in \Aut(\Gamma)$. Indeed one has to show that the relations defining $\Gamma$ are stable under the action of $\sigma$, which is obvious for the relations $[x,z]=1$ and $[y,z]=1$. 
  Let us look at the last relation. First, as $\Gamma$ is uniform,  let us recall that  $\Gamma_{n+1}=\Gamma^{p^n}$.
  We have: $$\sigma([x,y])=\sigma(xyx^{-1}y^{-1})\equiv x^{\zeta^{(n)}}y^{\zeta^{(n)}} x^{-\zeta^{(n)}}y^{-\zeta^{(n)}} (\mod \ \Gamma_n)
   \equiv z^{\left(\zeta^{(n)}\right)^2 p^s} (\mod \ \Gamma_n) \longrightarrow_{n} z^{\zeta^2 p^s}.$$
   
  To finish, let us show that the automorphism $\sigma$ is fixed-point-free.  Indeed the eigenvalues of the action of  $\sigma$ on the $\fq_p$-vector space $\Gamma/\Gamma_2\simeq \fq_p^3$  are 
   $\zeta$ (with multiplicity $2$) and $\zeta^2$, so $\act{\Gamma/\Gamma_2}$ is fixed-point-free and we conclude with Proposition \ref{actionuniforme}.
 \end{proof}

 \begin{rema}
  If $p\not \equiv 1 (\mod \ 3)$, there is no element $\sigma \in \Aut(\Gamma(s))$ of order $3$ with fixed-point-free action. Indeed, let $\sigma \in \Aut(\Gamma(s)/\Gamma(s)_2) \simeq  \Gl_3(\fq_p)$,  and suppose that $\sigma$ is fixed-point-free. The eigenvalues of $\sigma$ in $\overline{\fq_p}$ are $\zeta$ and/or $\zeta^2$ with multiplicity. But, as the trace of $\sigma$ is in $\fq_p$, then necessarily $\zeta \in \fq_p$ and $3 ~|~ p-1$.  
 \end{rema}

\begin{coro}
Assume $p \equiv 1 (\mod \ 3)$ and $p$ regular. For each $s \in \Nat$, there exist $\Gamma(s)$-extensions of numbers fields with arbitrarily large $\mu$-invariant.
\end{coro}
 
 \begin{rema}
 For $p=37$, which is the smallest irregular prime, we may not resort to the construction above with $\K=\Q(\zeta_{37})$, but we can still realize $\Gamma(s)$-extensions with arbitrarily large $\mu$-invariant by applying Theorem \ref{main} for $\K/\K_0=\Q(\zeta_7)/\Q(\sqrt{-7})$.
 \end{rema}

 \subsection{Counting the split primes in uniform extensions}
 Let $\L/\K$ be a Galois extension with $p$-adic analytic Galois group  $\Gamma$ of dimension $d$. Denote by $\Sigma$ the set of primes of  $\O_\K$ unramified in $\L/\K$. Assume $\Sigma$ finite. For   $\p \in \Sigma$, let $\C_\p$ be the conjugacy class of the Frobenius of $\p$ in $\L/\K$ and put $$\pi_{\L/\K}^{\mathrm{split}}(x)=\big|\{\p\in \Sigma, \ \C_\p=\{1\}, \N(\p) \leq x\} \big|,$$
 where $\N(\p):=\big|\O_\K/\p\big|$.
 
In \cite[Corollary 1 of Theorem 10]{Serre}, under GRH, Serre shows that for all $\varepsilon > 0$, $\pi_{\L/\K}^{\mathrm{split}}(x) = \bigO{x^{\frac{1}{2}+\varepsilon}}$. Without assuming GRH,  $\pi_{\L/\K}^{\mathrm{split}}(x) = \bigO{x/\log^{2-\varepsilon}(x)}$.

 \begin{prop}
Let $p$ be a regular prime and let  $\Gamma$ be a  uniform pro-$p$ group having a non-trivial fixed-point-free automorphism $\sigma$ of prime order $\ell ~|~ p-1$. 
Then there exists a constant $C>0$ and a $\Gamma$-extension $\L$ over a number field $\K$ such that
$$\pi_{\L/\K}^{\mathrm{split}}(x) \geq  C \frac{x^{1/\ell}}{\log x}, \qquad x\gg 0.$$
\end{prop}
 
 \begin{proof} One use the construction of section \ref{cadregeneral}. Put $\K=\Q(\zeta_{p^n})$ where $n$ is the smallest integer such that 
 $p^n(p-1)\geq 2\ell d$.
 Let $\K/\K_0$ be the unique cyclic extension of degree $\ell$; $\Gal(\K/\K_0)=\langle \sigma \rangle$.   Let $\L/\K$ be a uniform Galois extension with group $\Gamma$ constructed by the method of Proposition \ref{existence}.
 Let $\E(\L/\K)$ be the set of conjugacy classes of  $\Gal(\K/\Q)$ of the elements $\sigma^i$ with  $(i,\ell)=1$.  
 By Proposition  \ref{splitting}, a prime $\q$ of $\K$ for which the Frobenius conjugacy class $\sigma_\q \in \K/\Q$ is in $\E(\L/\K)$, splits totally in $\L/\K$. Hence
 $$
\begin{array}{rcl} \pi_{\L/\K}^{\mathrm{split}}(x) & \geq &\big|\{ \q \in \Sigma, \  \sigma_\q \in \E(\L/\K), \ \N(\q) \leq x\}\big| \\ 
&=& \big|\{ \q \in \Sigma, \exists (i,\ell)=1, \ \C_q(\K/\Q) = \sigma^i, \  q^\ell \leq x\} \\
&=& \big| \{ \q \in \Sigma, \exists (i,\ell)=1, \ \C_q(\K/\Q) = \sigma^i, \  q \leq x^{1/\ell}\}\big|
\end{array}$$
where $\C_q(\K/\Q)$ is the conjugacy class of the Frobenius of $q$ in $\Gal(\K/\Q)$.
By using Chebotarev Density Theorem, one concludes that 
$$\pi_{\L/\K}^{\mathrm{split}}(x) \gg  \frac{x^{1/\ell}}{\log x} \cdot$$
 \end{proof}

\begin{exem}
For the uniform group $\Gamma(s)$ of dimension 3 discussed in \S \ref{exemplenoncommutatif}, one obtains $\Gamma(s)$-extensions $\L/\K$ with
$\displaystyle{\pi_{\L/\K}^{\mathrm{split}}(x) \gg  \frac{x^{1/3}}{\log x} }$.
\end{exem}

\subsection{On $p$-rational fields with splitting}\label{splitting-section}
For all this section, assume that $p>2$.

By making use of fixed-point-free automorphisms, we have produced some uniform extensions with infinitely many totally split primes.  
Using Class Field Theory it is also possible to produce some free-pro-$p$ extensions with some splitting phenomena, but \emph{we do not know if it is possible to construct free pro-$p$ extensions in which infinitely many primes of the base field split completely}. 
In fact this question is related to the work of Ihara \cite{Ihara} and the recent work of the authors \cite{Hajir-Maire}.

\medskip

Let us show how to produce some free pro-$p$ extensions with some splitting.
Let $S$ and $T$ be two finite sets of places of number field $\K$ such that $S\cap T =\emptyset$. 
Recall that  $S_p=\{\p \in \K: \p|p\}$. Let $(r_1,r_2)$ be the signature of $\K$.
Let $\K_S^T$ be the maximal pro-$p$ extension of  $\K$ unramified outside $S$ and totally decomposed at $T$.
Put $\GST = \Gal(\K_S^T/\K)$;  $\A_S^T:=\GST^{\ab}$ and $\A_S^T/p:= \GST^{\ab}/\GST^{ab,p}$. Of course, one has $d_p \A_S^T = d_p \G_S^T$ and for $S=T=\emptyset$, $\A_\emptyset^\emptyset(\K)=\A(K)$.

\

We are interested in constructing an example where the group $\GST$ is a free pro-$p$ group and  $T$ is not empty. 

\

The following is a fundamental and classical result about the Euler characteristics of $\GST$.

\begin{prop}[Shafarevich and Koch] \label{koch-shafarevich} Suppose that $\K$ contains the $p$-roots of the unity.
Then, $$\d_p H^2(\GST,\fq_p) \leq d_p \A_T^S + |S| - 1$$
and $$d_p \GST = d_p \A_T^S+ |S| - |T| -(r_1+r_2) + \sum_{v\in S \cap S_p}[\K_v:\Q_p]\cdot$$
\end{prop}

\begin{proof}
See   for example \cite[Corollary 3.7.2, Appendix]{gras}  for the bound for the $H^2$ and \cite[Theorem 4.6, Chapter I, \S 4]{gras} for the $H^1$.
\end{proof} 

The first  inequality of the previous proposition  would allow us   to produce a free pro-$p$ extension with complete splitting at the primes in $T$ if $d_p A_T^S +|S| - 1=0$.  
On the other hand, if we apply the second line , let us read the second of Proposition \ref{koch-shafarevich} with the role of $S$ and $T$ reversed, we find that $$d_p \A_T^S \geq |T|-(r_1+r_2+|S|).$$ We conclude that, under this strategy, $|T| \leq r_1 + r_2 + 1,$ so this method is rather limited in scope.   

\medskip

Let us take $\K=\Q(\zeta_p)$ when $p$ is regular and $S=S_p$. Then the group $\GST$ will be pro-$p$ free when $\A_T^S$ is trivial. 
As the $p$-class group of $\K$ is trivial, one has:
$$\A_T^S \simeq (\prod_{v\in T} \U_v)/\O_\K^S,$$
where $\U_v$ is the pro-$p$ completion of the local units at $v$ and where $\O_K^S$ is the group of $S$-units of $\K$.
Now, there are several types of scenarios where the quotient $(\prod_{v\in T} \U_v)/\O_\K^S$ is trivial. Let us give one. Suppose that $T=\{\ell\}$ where $\ell$ is inert in $\K/\Q$. Then $\U_\ell $ is isomorphic to the $p$-part of $\fq_\ell^\times$. Hence, the global $p$th roots of unity will kill this part when $|\U_\ell|=p$ i.e.~ when  $\ell^{p-1}-1$ is exactly divisible by $p$. When this is the case, the group $\GST$ is free  on $(p-3)/2$ generators.

\medskip

\subsection{Incorporating the Galois action}
We can take these ideas a bit further by studying the Galois action.
Throughout this subsection, we fix the following notation and assumptions (we still assume $p>2$).
Let $\K_0$ be a number field, $\K/\K_0$ a cyclic extension of integer $m$ co-prime to $p$. 
Put  $\Delta=\Gal(\K/\K_0)$.
Let $r=r_1(\K_0)+r_2(\K_0)$ be the number of archimedean places of $\K_0$.
Let $S$ and $T$ be two finite sets of places of $\K_0$ such that $S\cap T =\emptyset$ and $S$ contains $S_p$. By abuse of notation, the set of places of  $\K$ above $S$ and $T$ are again called $S$ and $T$ respectively.
Denote by  $S_{\split}$ (resp. $T_{\split}$) the set of places of $S$ (resp. of $T$) splitting in
 $\K/\K_0$ an by  $S_{\inert}$ (resp. $T_{\inert}$)  the set of places of $S$ (resp. of $T$) not splitting in  $\K/\K_0$.
As in \S \ref{section-p-rational}, the arithmetic objects of interest have a structure  as $\fq_p[\Delta]$-modules. 
We recall the following mirror identity from the book of  Gras \cite[Chapter II, \S 5.4.2]{gras2}:

\begin{theo}\label{caracteres} Assume that $S_p \subset S$ and that  $\K$ contains a primitive $p$th root  of unity. Then:
 $$\omega \chi^{-1}(\A_S^T)-\chi(\A_T^S)= r \chi_{\mathrm{reg}} + \omega -\1 + |S_{\inert}| \1 + |S_{\split}| \chi_{\mathrm{reg}} - |T_{\split}|\chi_{\mathrm{reg}} - |T_{\inert}| \omega.$$
\end{theo}


 \begin{prop} \label{cyclotomicextension}
 
 Suppose that $\K$, which is a cyclic degree $p$ extension of $\K_0$,  contains a primitive $p$th root of unity and is $p$-rational.  Then 
 $K_0$, as well as every number field $\K_n$ in the cyclotomic extension  $\K_\infty$ of $\K$ is also $p$-rational.
 
  
 \end{prop}
 
 \begin{proof}
It is an obvious extension of  Proposition \ref{p-rational}.
 \end{proof}

 Now we assume that $\K_0$ is totally imaginary  and we take $S=S_p$.

By hypothesis, there is no abelian unramified  $p$-extension of  $\K$  in which  $p$ splits completely.

Then for $T=\emptyset$, by Theorem \ref{caracteres}, one has:
 $$\omega \chi^{-1}(\A_S)= r \chi_{\mathrm{reg}} + \omega.$$
 The group  $\A_S(\K_0)$ corresponds by Class Field Theory to the Galois group of the maximal \emph{abelian} pro-$p$ extension of $\K_0$ unramified outside $S$. The Galois group  $\G_S(\K_0)$ is free on $r+1$ generators.
 The action of $\Delta$ on  $\A_S/p$ being semi-simple, the $\fq_p[\Delta]$-module $\A_S(\K_0)/p$ is isomorphic to  $\A_S^T/p$ on which the $\Delta$ action  is trivial.
 
 \
 
By the Chebotarev Density Theorem, we can choose a set  $T:=\{\p_0,\cdots, \p_r\}$ of places of  $\K_0$ all inert in $\K/\K_0$, with $|T|=d_p A_S(\K_0)=r+1$, such that 
 the Frobenius symbols 
 $\sigma_{\p_i}$, $i=0,\cdots, r$,  generate the $p$-group $\A_S(\K_0)\otimes \fq_p$.
 By the choice of $T$, one has $\chi(\A_S^T)=\chi(\A_S)- |T| \1$.
 Thanks to Theorem \ref{caracteres} one then obtains:
 $$ \chi(\A_T^S)= \omega \chi^{-1}(\A_S^T) -\omega \chi^{-1}(\A_S) - |T| \omega=0.$$
 Then the maximal  pro-$p$-extension of $\K$ unramified outside $T$ and splitting at $S$, is trivial.
One then uses Proposition  \ref{koch-shafarevich} to obtain:

\begin{prop} \label{caractere2} Under the conditions of 
Proposition \ref{cyclotomicextension}, and with $S=S_p$ and $T$ as above, the pro-$p$-group $\GST$ is free on   $r(\ell-1)$ generators and $$\chi(\A_S^T)= r(\chi_{\mathrm{reg}}-\1).$$
\end{prop}

\begin{rema}
 $(i)$ The main point of Proposition \ref{caractere2} is that the action of $\sigma$ on $\A_S^T$ is fixed-point-free, where $\Delta=\langle \sigma \rangle$.
 
 $(ii)$ By Proposition \ref{cyclotomicextension}, the degree $r$ can be taken arbitrarily large. 
\end{rema}

As corollary, one obtains:

\begin{coro}
Assume the conditions of Proposition \ref{caractere2}. Let  $\L/\K_0$ be a sub-extension of   $\K_S^T/\K_0$ such that  $\Gamma = \Gal(\L/\K)$ is a uniform pro-$p$ group. Then every (totally) inert (odd) prime $\p \subset \O_\K$ in $\K/\K_0$ splits completely in $\L/\K$.
\end{coro}

\subsection{Some heuristics on  $p$-rationality}

Here we exploit some conjectures on $p$-rationality to get further heuristic evidence in direction of Conjecture \ref{conjecture}.
These are of two types. The first is in the spirit of the Leopoldt Conjecture -- in a certain sense this is a transcendantal topic; the main reference is the recent work of Gras \cite{gras-heuristics}.
The second one is inspired by the heuristics of Cohen-Lenstra \cite{Cohen-Lenstra}.

\begin{conjecture}[Gras, Conjecture 8.11 of \cite{gras-heuristics}]\label{conjectureregulator}
Let $\K$ be a number field. Then for  $p \gg 0$, the field $\K$ is $p$-rational.
\end{conjecture}

\begin{prop}\label{enfamille} Let $\PP$ be an infinite set of prime numbers and $m$ an integer  co-prime to all $p\in \PP$.
Let $(\Gamma_p)_{p \in \PP}$ be a family of uniform pro-$p$ groups all of dimension $d$ and all having a fixed-point-free automorphism of order $m$.
If Conjecture \ref{conjectureregulator} is true for some number field $\K$ of degree  $m n$, then for all but finitely many primes $p\in \PP$, the groups $\Gamma_p$ have arithmetic realizations as Galois groups with arbitrarily large $\mu$-invariant.
\end{prop}

\begin{proof}
Let us take $\K_0$ be a totally imaginary abelian number field of degree $n$ over $\Q$ and let $\K_1/\Q$ be a degree $m$ cyclic extension  with $\K_0\cap \K_1=\emptyset$. Let $\K=\K_0\K_1$ be the compositum.
Put $\langle \sigma \rangle=\Gal(\K/\K_0)$. Under Conjecture \ref{conjectureregulator}, one knows that for large $p$ depending only on $\K$, the field $\K$ is $p$-rational. One then applies Theorem \ref{main}.
\end{proof}

\begin{exem} Let us fix $s\in \Nat$ and let $\PP$ be the set of primes $p \equiv 1 \bmod 3$.  Let $\Gamma(s,p)$ be the pro-$p$ group $\Gamma(s)$ of dimension $3$ introduced in \S \ref{exemplenoncommutatif}.
One can apply Proposition \ref{enfamille} to the family of groups $\left(\Gamma(s,p)\right)_{p\in \PP}$.
\end{exem}

\medskip

Let us consider another idea for studying $p$-rationality, in the spirit of the Cohen-Lenstra Heuristics \cite{Cohen-Lenstra}.
For a prime number $p\geq 5$ and an integer $m>1$ not divisible by $p$,  let us consider the family $\Ff_m(\K_0)$ of cyclic extensions of degree $m$ over a fixed $p$-rational number field $\K_0$.
 Following the idea of Pitoun and Varescon \cite{Pitoun-Varescon}, it seems reasonable to make the following conjecture:

 \begin{conjecture}\label{conjecture-CL} 
 Let us fix $p$ and $m$ as above.
 If $\K_0$ is a $p$-rational quadratic imaginary field,  
 the proportion of  number fields $\K$ in $\Ff_{m}(\K_0)$ which are $p$-rational (counted according to increasing absolute value
of the discriminant) is positive.
\end{conjecture}

Obviously, Conjecture \ref{conjecture-CL} implies Conjecture \ref{conjecture}.


\end{document}